\documentclass[a4paper,10pt, reqno]{amsart}

\usepackage{stmaryrd}
\usepackage{pdfsync}
\usepackage[numbers]{natbib}
\usepackage{float}
\usepackage{graphicx}
\usepackage{url}
\usepackage{amssymb,amsmath,amsthm,verbatim,longtable,amscd,latexsym}
\usepackage{mathabx}
\usepackage[all, ps]{xy}
\usepackage{color}
\usepackage{hyperref}

\renewcommand\epsilon{\varepsilon}
\renewcommand\phi{\varphi}

\newcommand{\amod}{A\mbox{-}\mathsf{Mod}}

\newcommand{\lmapsto}{\longmapsto}
\newcommand{\lto}{\longrightarrow}

\newcommand{\op}{\mathrm{op}}

\newcommand{\C}{\mathcal C}

\newcommand{\id}{\mathsf{id}}

\newcommand{\ma}{\mathcal A}

\newtheorem{thm}{Theorem}[section]
\newtheorem{prop}[thm]{Proposition}
\newtheorem{lem}[thm]{Lemma}
\newtheorem{cor}[thm]{Corollary}

\theoremstyle{definition}
\newtheorem{exa}[thm]{Example}
\newtheorem{defn}[thm]{Definition}

\theoremstyle{remark}

\newtheorem{rem}[thm]{Remark}

\title{Factorisations of distributive laws}
\author{Ulrich Kr\"ahmer}
\author{Paul Slevin}

\begin{document}
\begin{abstract}
Recently, B\"ohm and \c Stefan constructed duplicial
(paracyclic) objects from distributive laws between
(co)monads. Here
we define the category of factorisations of a
distributive law,
show that it acts on this construction, and
give some explicit examples.
\end{abstract}
\maketitle

\tableofcontents
\section{Introduction}
Distributive laws between monads were originally
defined by Beck in~\cite{MR0241502} and correspond to
monad structures on the composite of the two monads.
They have found many applications in mathematics as
well as computer science; see
e.g.~\cite{MR2520969,MR2504663,MR1692751,MR2220892}. 

Recently, distributive laws have been used by
B\"ohm and \c Stefan \cite{MR2415479} to construct new
examples of duplicial objects \cite{MR826872}, and hence cyclic
homology theories.
The paradigmatic example of such a theory is the cyclic
homology $HC(A)$ of an
associative algebra $A$ \cite{MR823176,MR695483}. 
It was
observed by Kustermans, Murphy, and Tuset \cite{MR1943179}
that the functor $HC$ can be twisted by automorphisms
of $A$. The aim of the present paper is to extend this
procedure to any duplicial object defined by a
distributive law.

Given a distributive law $ \chi $ we define in
Section~\ref{factdef} the category
$\mathcal{F}(\chi)$ of its \emph{factorisations}. 
The main technical results are the definition of a
monoidal structure on $\mathcal{F}(\chi)$
(Lemma~\ref{welldefined} and Proposition~\ref{strict
monoidal}), a characterisation of the comonoids in
$\mathcal{F}(\chi)$ (Proposition~\ref{chi-comonoids}), and
the definition of actions of 
$\mathcal{F}(\chi)$ on
the category of admissible data (septuples in \cite{MR2415479})
which turns the latter into an $\mathcal{F}(\chi)$-bimodule category 
(Theorem~\ref{main theorem} and Corollary~\ref{korollar}).

The remainder of the paper is devoted to examples. We
begin by considering factorisations of distributive
laws on
Eilenberg-Moore categories, interpreting these as flat
connections (Section~\ref{flat}). In particular, we 
present the twisting of cyclic homology in
this framework (Section~\ref{connectexample}). We then
describe examples arising from Hopf algebras
(Section~\ref{mixed}). The final examples are concerned
with BD-laws, braidings (Section~\ref{braided}), and
quantum doubles of Hopf algebras
(Section~\ref{qd}).

Throughout this paper, $\mathcal{A,B,C}\ldots$ are
categories, $A,B,C,\ldots$ are functors, and greek
letters are used to denote natural transformations. We
use $ \circ $ to denote composition of morphisms and
vertical composition of natural transformations. The
composition of functors and the horizontal composition of natural transformations will
be denoted simply by concatenation. The identity
morphism, functor and natural transformation is
denoted by $\id$. However, we denote the horizontal
composition $ \alpha \,\id_A \beta$ by $ \alpha A \beta
$.

\subsection*{Acknowledgements}
UK is supported by the EPSRC First Grant ``Hopf
Algebroids and Operads'' and the 
Polish Government Grant 2012/06/M/ST1/00169. 
PS is supported by an EPSRC Doctoral
Training Award. 

\section{Preliminaries}
In this section, we recall basic definitions and results that
are needed later.

\subsection{(Co)monads}
Let $\mathcal A$ be a category.
\begin{defn}
A $\emph{comonad on $\ma$}$ is a triple $\mathbb C = (C,
\Delta, \epsilon)$
where $C$ is an endofunctor on $\mathcal A$, and $\Delta \colon C \lto CC$ and
$\epsilon \colon C \lto \mathsf{id}_\ma$ are natural
transformations such that
$$
	C \Delta \circ \Delta =
	\Delta C \circ \Delta,\quad
	\varepsilon C \circ \Delta = \id_C =
	C\varepsilon  \circ \Delta,
$$
that is, the
two diagrams
$$
\xymatrix{
	C \ar[r]^-\Delta \ar[d]_-\Delta & CC \ar[d]^-{C \Delta} \\
	CC \ar[r]_-{\Delta C} & CCC 
	}
	\quad\quad\quad
\xymatrix{
	C \ar[r]^-{\Delta} \ar[d]_-{\Delta} \ar@{=}[dr] & CC \ar[d]^-{C \epsilon} \\
	CC \ar[r]_-{\epsilon C} & C 
}
$$
commute.
\end{defn}

In other words, a comonad is a comonoid (or coalgebra) in
the monoidal category $[\ma,\ma]$ of endofunctors on $\ma$ 
(with composition as tensor product).
Dually, a \emph{monad} on a category $\mathcal C$ is a monoid
(algebra) in $[\mathcal C,\mathcal C]$. 

\subsection{Module categories} 
Next, we recall the
notion of a module category (also known as an
$\mathcal M$-category) over a monoidal
category $(\mathcal M, \otimes, \mathbf 1)$.
For the purpose of this paper, all monoidal categories
and their module categories are strict, and
by abuse of notation we will write
$\mathcal M$ to refer to the whole triple $(\mathcal
M,\otimes,\mathbf 1)$.

\begin{defn}
A \emph{left module category for} $\mathcal M$ is a
pair $(\mathcal C, \rhd)$ where $\mathcal C$ is a
category and $\rhd \colon \mathcal M \times \mathcal C
\lto \mathcal C$ is a functor such that we have
functorial identities
\begin{align*}
\mathbf 1 \rhd P &= P &\text{and} & &X \rhd (Y \rhd P) &= (X \otimes Y) \rhd P.
\end{align*}
for all objects $X,Y$ in $\mathcal M$ and $P$ in $\mathcal C$. We call $\rhd$ the \emph{left action of $\mathcal M$ on $\mathcal C$}.
\end{defn}

Dually, one defines a \emph{right module category} $(\mathcal D, \lhd)$.
A \emph{bimodule category} is a triple $(\mathcal C,
\rhd, \lhd)$ where $(\mathcal C, \rhd)$ and $(\mathcal
C,\lhd)$ are right respectively left module categories and the actions commute, i.e.\ for all objects $X,Y$ in $\mathcal M$ and $P$ in $\mathcal C$ 
 we have
\[
	X \rhd (P \lhd Y) = (X \rhd P) \lhd Y,
\]
again functorially in $X,Y$ and $P$.
We immediately have the following.
\begin{lem}\label{product bimodule}
Let $(\mathcal C, \rhd)$ and $(\mathcal D, \lhd)$ be
left respectively right module categories. 
Then $\mathcal C \times \mathcal D$ is a bimodule category 
with actions given by
$$
	X \rhd (P,Q) \lhd Y = (X \rhd P, Q \lhd Y)
$$
for all objects $X,Y$ in $\mathcal M$, $P$ in $\mathcal C$ and $Q$ in $\mathcal D$.
\end{lem}

\subsection{Eilenberg-Moore categories}\label{Eilenberg-Moore}
The comonads we are mostly interested in arise as
restrictions of monads to their Eilenberg-Moore
categories.

\begin{defn}
Let $(\mathcal C, \rhd)$ be a left module category for
a monoidal category $\mathcal M$, and let $\mathbb B = (B, \mu, \eta)$ be a monoid in $\mathcal M$. The \emph{Eilenberg-Moore category of $\mathbb B$}, denoted by $\mathcal C^\mathbb B$, is the category whose objects are pairs $(X, \alpha)$, where $X$ is an
object of $\mathcal C$ 
and $\alpha \colon B \rhd X \lto X$ is a morphism in
$\mathcal C$ 
such that the diagrams
$$
\xymatrix{
	\ar[dr]_-{\mu \rhd \mathsf{id}_X}(B \otimes B) \rhd X
	\ar@{=}[r] & B\rhd (B \rhd X)
	\ar[r]^-{\id_B \rhd \alpha}
	& B \rhd X  \ar[d]^-\alpha\\
& B \rhd X \ar[r]_-\alpha & X
}
\quad\quad\quad
\xymatrix{
	\mathbf 1 \rhd X \ar@{=}[dr] \ar[r]^{\eta \rhd
\mathsf{id}_X} & B \rhd X \ar[d]^-\alpha \\
& X
}
$$
commute, and whose morphisms $f \colon (X, \alpha) \lto
(X', \alpha')$ are morphisms $f \colon X \lto X'$ in
$\mathcal C$ 
such that the diagram
$$
\xymatrix{
B \rhd X \ar[d]_-\alpha \ar[r]^-{\id_B \rhd f} & B \rhd X' \ar[d]^-{\alpha '} \\
X \ar[r]_-f & X'
}
$$
commutes.
\end{defn}

Now observe that the monoid $\mathbb B$ defines a comonad
$\tilde{\mathbb B} = (\tilde B, \tilde \Delta, \tilde \epsilon)$ on
$\ma=\mathcal C^\mathbb B$ where $\tilde B$ is defined on objects and 
morphisms by
\begin{align*}
	\tilde B(X, \alpha) &= (B \rhd X, \mu\rhd
	\mathsf{id}_X), &  \tilde B(f) &= \id_B \rhd f,
\end{align*}
and $\tilde \Delta, \tilde\epsilon$ are defined on 
objects $(X, \alpha)$ by
\begin{align*}
\xymatrix@C=5em{
B \rhd X = B \rhd (\mathbf 1 \rhd X) \ar[r]^-{\id_B \rhd (\eta \rhd \mathsf{id}_X)} & B \rhd (B \rhd X)
}
\quad\quad\quad
\xymatrix{
B \rhd X \ar[r]^-{\alpha}&X
}
\end{align*}
respectively.

In particular, every category $\mathcal C$
is in an obvious way a module category over
$[\mathcal C,\mathcal C]$. In this case, our definition
of Eilenberg-Moore category of a monad $\mathbb B$ on
$\mathcal C$ is the same as the usual
definition~\cite[p. 139]{MR1712872}.

\subsection{Distributive laws}\label{distributive}
Next we define distributive laws. Note that we consider
them between (co)monads and arbitrary endofunctors as
is common in the computer science
literature, see e.g.~\cite{MR1692751}.

\begin{defn}
Let $\mathbb T = (T, \Delta, \epsilon)$ be a comonad on
$\mathcal A$ and let $C$ be an endofunctor on $\mathcal A$. A \emph{distributive law between the comonad} $\mathbb T$ \emph{and the endofunctor} $C$ is a transformation $\chi \colon TC \lto CT$ such that the two diagrams
$$
\xymatrix{
\ar[d]_-{\Delta C} TC \ar[r]^-{\chi} &  CT  
\ar[r]^-{C \Delta} & CTT  \\
TTC \ar[rr]_-{T \chi} && TCT \ar[u]_-{\chi T}
} \quad\quad\quad
\xymatrix{
TC \ar[r]^-\chi \ar[dr]_-{\epsilon C} & CT \ar[d]^-{C \epsilon} \\
& C
}
$$
commute. We denote this by $\chi \colon \mathbb T \lto C$. 
Analogously, we define a distributive law $\chi \colon T \lto \mathbb C$ between an endofunctor $T$ and a comonad $\mathbb C$. A \emph{comonad distributive law} $\chi \colon \mathbb T \lto \mathbb C$ is a transformation $\chi$ which is a distributive law between endofunctors and comonads in both ways.
\end{defn}

Dually, we can define distributive laws involving
monads; distributive laws from a monad to a comonad
are usually called mixed distributive laws. 

One application of distributive laws is to lift
endofunctors to Eilenberg-Moore categories: 
let $\mathbb B$ be a monad on a category $\mathcal C$ and
$\theta \colon \mathbb B \lto D$ be a 
distributive law. We define a functor
$\tilde D \colon \mathcal C^\mathbb B \lto \mathcal C^\mathbb B$ 
as follows.
On objects we define
$$
	\tilde D(X, \alpha) = (DX, D\alpha \circ \theta_X)
$$
and we define $\tilde Df = Df$ on morphisms. The
distributive law $\theta$ lifts to give one $\theta
\colon\tilde{\mathbb B} \lto \tilde D$ where
$\tilde{\mathbb B}$ is the comonad described in Section~\ref{Eilenberg-Moore}. If $D$ is part of a comonad $\mathbb D = (D, \Delta, \epsilon)$, and $\theta$ is a mixed 
 distributive law $\mathbb B \lto \mathbb D$, then $\tilde D$ is part of a comonad
$$
\tilde {\mathbb D} = (\tilde D, \Delta, \epsilon)
$$
and $\theta$ lifts to a comonad distributive law $ \theta \colon\tilde{\mathbb B} \lto \tilde{\mathbb D}$.

See~\cite{MR0241502,MR0323864} for more details on distributive laws.

\subsection{The categories of $\chi$-coalgebras}
Let $\mathbb T = \left(T, \Delta^T, \epsilon^T\right)$ and $\mathbb C = \left(C, \Delta^C, \epsilon^C\right)$ be comonads on $\mathcal A$, and let $\chi \colon \mathbb T \lto \mathbb C$ be a distributive law.
\begin{defn}
A \emph{right $\chi$-coalgebra} is a triple $(M,\mathcal X, \rho)$ where $\mathcal X$ is a category, $M \colon \mathcal X \lto \mathcal A$ is a functor and $\rho \colon TM \lto CM$ is a natural transformation such that the diagrams
$$
	\xymatrix{
	TM \ar[r]^-{\Delta^T M} 
	\ar[d]_-{\rho} & TTM \ar[r]^{T\rho} & TCM \ar[d]^-{\chi_M} \\
CM \ar[r]_-{\Delta^C M} 
& CCM & CTM \ar[l]^-{C\rho}
}
\quad\quad\quad
\xymatrix{
& TM \ar[dl]_-{\epsilon^T M}
\ar[d]^-{\rho}\\
M & CM \ar[l]^-{\epsilon^C M} 
}
$$
commute. A \emph{morphism of right $\chi$-coalgebras} between $(M, \mathcal X, \rho)$ and $(M', \mathcal X', \rho')$ is a pair $({\varphi},F)$, where $F \colon \mathcal X \lto \mathcal X'$ is a functor and ${\varphi} \colon M  \lto M'F$ is a natural transformation such that the diagram
$$
\xymatrix{
TM \ar[r]^-{T{\varphi}} \ar[d]_-\rho & TM'F\ar[d]^-{\rho ' F}\\ 
CM \ar[r]_-{C{\varphi}} & CM'F
}
$$commutes. We define composition of morphisms by
$$
({\varphi}', F') \circ ({\varphi}, F) = ({\varphi}' F \circ {\varphi}, F'F) 
$$ and we define identity morphisms by $\mathsf{id}_{(M, \mathcal X, \rho)} = (\mathsf{id}_M, \mathsf{id}_\mathcal X)$. We denote the category of right $\chi$-coalgebras by $\mathcal R(\chi)$.
\end{defn}

Dually, we define the category
$\mathcal L(\chi)$ of \emph{left $\chi$-coalgebras}
$(N, \mathcal Y, \lambda)$.

\subsection{The construction of B\"ohm and \c Stefan}
Finally, we recall the construction of duplicial
functors from a comonad distributive law
$\chi \colon \mathbb T \lto \mathbb C$ on a
category $\mathcal A$ due to B\"ohm and \c
Stefan.

\begin{defn}
 The category of \emph{admissible data over} $\chi$ is the product category
$$
\mathcal S(\chi) := \mathcal R (\chi) \times \mathcal L (\chi).
$$
Admissible data are called \emph{admissible septuples} in~\cite{MR2415479}.
\end{defn}

To every admissible datum $(M, \mathcal X, \rho, N, \mathcal Y, \lambda)$ there is an associated duplicial functor $\mathcal X \lto \mathcal Y$ defined by
$$
D_\bullet (M, \mathcal X, \rho, N, \mathcal Y, \lambda) = N  T^{\bullet + 1} M
$$
which is given objectwise by taking the bar resolution of $M$ with respect to the comonad $\mathbb T$, and then applying the functor $N$. If $\mathcal Y$ is an abelian category, we can apply the duplicial functor to an object $X$ in $\mathcal X$ resulting in a duplicial object in $\mathcal Y$ of which we can take the cyclic homology.

This construction, which unifies and generalises the
definition of the cyclic homology of an associative
algebra as well as Hopf-cyclic homology, is detailed
in~\cite{MR2415479} for the case that
$\mathcal X = \{ * \}$ is the terminal category.

\section{Theory}\label{results}
\subsection{The category of factorisations $\mathcal F(\chi)$}\label{factdef}
Throughout this section, let $\mathbb T = \left(T, \Delta^T, \epsilon^T\right)$ and $\mathbb C = \left(C, \Delta^C, \epsilon^C\right)$ be comonads on a category $\mathcal A$, and let $\chi \colon \mathbb T \lto \mathbb C$ be a distributive law.
The main definition of the present paper is the
following:

\begin{defn}
A \emph{factorisation of} $\chi$ is a triple
$(\Sigma, \sigma, \gamma)$ where
$\Sigma$ is an endofunctor on $\mathcal A$, and $\sigma \colon \mathbb T \lto \Sigma$ and $\gamma \colon \Sigma \lto \mathbb C$ are
distributive laws satisfying
the \emph{Yang-Baxter} condition; that is, the hexagon
$$
\xymatrix@R=0.5em{
& \Sigma TC \ar[r]^-{\Sigma \chi} &
\Sigma CT \ar[dr]^-{\gamma T} &\\ 
T\Sigma C \ar[ur]^-{\sigma C} \ar[dr]_-{T\gamma} 
& & & C \Sigma T \\
& TC\Sigma \ar[r]_-{\chi{\Sigma}} & CT\Sigma
\ar[ur]_-{C \sigma} &
}
$$
commutes. A \emph{morphism} $\alpha \colon (\Sigma, \sigma, \gamma) \lto (\Sigma', \sigma', \gamma')$ of factorisations is a natural transformation $\alpha \colon \Sigma \lto \Sigma'$ which is compatible with $T$ and $C$ in the sense that the diagrams
$$
\xymatrix{
T\Sigma \ar[r]^-{T \alpha} \ar[d]_-{\sigma} &
T \Sigma ' \ar[d]^-{\sigma '} \\
\Sigma T \ar[r]_-{\alpha T} & \Sigma ' T 
}
\quad \quad \quad
\xymatrix{
\Sigma C \ar[d]_-{\gamma} \ar[r]^-{\alpha C} 
& \Sigma ' C\ar[d]^-{\gamma '} \\
C \Sigma \ar[r]_-{C\alpha} & C \Sigma '
}
$$
commute. There are identity morphisms $\mathsf{id}_{(\Sigma, \sigma, \gamma)} = \mathsf{id}_\Sigma$, and composition of morphisms is given by the vertical composite. This defines the \emph{category of factorisations} which we denote by $\mathcal F(\chi)$.
\end{defn}
Similarly, we may also define factorisations of a monad or mixed distributive law.
\subsection{The monoidal structure} 
We define a functor
$$
\otimes \colon \mathcal F(\chi) \times \mathcal F(\chi) \lto \mathcal F(\chi)
$$
as follows. On objects we define
$$
	(\Sigma, \sigma, \gamma)
	\otimes
	(\Sigma', \sigma', \gamma') =
	(\Sigma\Sigma', \Sigma \sigma' \circ \sigma{\Sigma'}, \gamma{\Sigma'} \circ \Sigma \gamma')
$$
and for two morphisms $\alpha, \beta$ we define $\alpha \otimes \beta$ to be $\alpha\beta$, the horizontal composite of the natural transformations.
\begin{lem}\label{welldefined}
The assignment $\otimes$ is a well-defined functor.
\end{lem}
\begin{proof}
Firstly, $\otimes$ is well-defined on objects if $\Sigma \sigma' \circ \sigma{\Sigma'}$ and $ \gamma{\Sigma'} \circ \Sigma \gamma'$ satisfy the Yang-Baxter condition. Consider the following diagram
$$
\xymatrix@R=0.9em{
 & & \Sigma \Sigma' TC \ar[r]^-{\Sigma \Sigma' \chi} &
\Sigma \Sigma'CT \ar[dr]^-{\Sigma \gamma' T} & & 
 \\& \Sigma T \Sigma' C \ar[ur]^-{\Sigma \sigma' C}
\ar[dr]_-{\Sigma T \gamma'} & & &\Sigma C \Sigma' T
\ar[dr]^-{\gamma{\Sigma'T}} & \\
T \Sigma \Sigma'C \ar[dr]_-{T \Sigma \gamma'}
\ar[ur]^-{\sigma{\Sigma ' C}} &
&\Sigma T C \Sigma' \ar[r]_-{\Sigma \chi{\Sigma'}}
& \Sigma CT \Sigma' \ar[ur]_-{\Sigma C \sigma'}
\ar[dr]_-{\gamma{T \Sigma'}} & & C \Sigma \Sigma' T \\
& T \Sigma C \Sigma' \ar[ur]_-{\sigma{C \Sigma'}} \ar[dr]_-{T \gamma{\Sigma'}} & &  & C \Sigma T \Sigma' \ar[ur]_-{C \Sigma{\sigma '}} \\
& &TC \Sigma \Sigma' \ar[r]_-{\chi{\Sigma\Sigma'}}  & CT \Sigma \Sigma' \ar[ur]_-{C \sigma{\Sigma' }} & &
}
$$
The left square commutes by naturality of $\sigma$ and
the right square commutes by naturality of $\gamma$.
The inner hexagons commute by the Yang-Baxter
conditions. Therefore, the outer hexagon commutes, so the required condition is satisfied.

Secondly, let
\begin{align*}
  \alpha \colon (\Sigma, \sigma, \gamma) &\lto (\Gamma, \kappa, \nu)  & &\text{and}  &\beta \colon (\Sigma', \sigma', \gamma') & \lto (\Gamma', \kappa', \nu')
\end{align*}
be morphisms in $\mathcal F(\chi)$. Consider the diagram
$$
\xymatrix{
T \Sigma \Sigma' \ar[r]^-{T \alpha{\Sigma'}} \ar[d]_-{\sigma{\Sigma'}} & T \Gamma' \Sigma' \ar[d]^-{\kappa{\Sigma'}} \ar[r]^-{T \Gamma \beta} & T \Gamma \Gamma' \ar[d]^-{\kappa{\Gamma'}} \\
\Sigma T \Sigma'  \ar[d]_-{\Sigma \sigma'} \ar[r]_-{\alpha{T\Sigma'}} & \Gamma T \Sigma' \ar[d]^-{\Gamma \sigma'} \ar[r]_-{\Gamma T\beta} & \Gamma T \Gamma' \ar[d]^-{\Gamma \kappa'} \\
\Sigma \Sigma' T \ar[r]_-{\alpha{\Sigma' T}} & \Gamma \Sigma' T \ar[r]_-{\Gamma \beta T} & \Gamma \Gamma' T
}
$$
The bottom-left square commutes by naturality of
$\alpha$, the top-right square commutes by naturality
of $\kappa$, and the two remaining inner squares
commute since $\alpha$ and $\beta$ are compatible with
$T$. Therefore, the outer square commutes and $\alpha
\otimes \beta$ is compatible with $T$. A similar
argument shows that $\alpha \otimes \beta$ is
compatible with $C$. It is clear that $\otimes$
respects composition of morphisms and identity
morphisms. Therefore, $\otimes$ is well-defined on morphisms.
\end{proof}
Let $\mathbf 1$ denote the trivial factorisation $(\mathsf{id}_\mathcal A,\mathsf{id}_T,\mathsf{id}_C)$.
\begin{prop}\label{strict monoidal}
The triple $(\mathcal F(\chi),\otimes,\mathbf 1)$ is a strict
monoidal category.
\end{prop}
\begin{proof}
It is clear that $\mathcal T \otimes \mathbf 1 = \mathbf 1 \otimes \mathcal T = \mathcal T$ for all factorisations $\mathcal T$. Consider the products of factorisations
\begin{align*}
&((\Sigma, \sigma, \gamma) \otimes (\Sigma', \sigma', \gamma')) \otimes (\Sigma'', \sigma'', \gamma'')\\
&= (\Sigma \Sigma', \Sigma \sigma' \circ \sigma{\Sigma'}, \gamma{\Sigma'} \circ \Sigma \gamma') \otimes (\Sigma'', \sigma'', \gamma'') \\
&= (\Sigma \Sigma' \Sigma'', \Sigma \Sigma'\sigma'' \circ \Sigma \sigma'{\Sigma''} \circ \sigma{\Sigma' \Sigma''}, \gamma{\Sigma' \Sigma''} \circ \Sigma \gamma'{\Sigma''} \circ \Sigma \Sigma' \gamma'')
\end{align*}
and
\begin{align*}
&(\Sigma, \sigma, \gamma) \otimes ((\Sigma', \sigma', \gamma') \otimes (\Sigma'', \sigma'', \gamma''))\\
&= (\Sigma, \sigma, \gamma) \otimes (\Sigma'\Sigma'', \Sigma'\sigma''\circ \sigma'{\Sigma''}, \gamma'{\Sigma''} \circ \Sigma' \gamma '')\\
&= (\Sigma \Sigma' \Sigma'', \Sigma \Sigma'\sigma'' \circ \Sigma \sigma'{\Sigma''} \circ \sigma{\Sigma' \Sigma''}, \gamma{\Sigma' \Sigma''} \circ \Sigma \gamma'{\Sigma''} \circ \Sigma \Sigma' \gamma'').
\end{align*}
These are equal so $\otimes$ is an associative tensor
product (observe that all equalities are functorial).
\end{proof}
\begin{rem}
If we ignore set theoretic issues, we can define a 2-category
$$
\underline{\mathsf{dist}} := \mathbf{Cmd} (\mathbf{Cmd}(\underline{\mathsf{CAT}^\op})^\op)
$$
where $\underline{\mathsf{CAT}}$ is the 2-category of categories, functors and natural transformations, $\mathbf{Cmd}$ denotes taking the 2-category of comonads, and $\op$ denotes reversal of 1-cells. The 0-cells of this 2-category are comonad distributive laws $\chi$
and we have
$$
\mathcal F(\chi) = \underline{\mathsf{dist}} (\chi, \chi)
$$
which is a strict monoidal category. This gives another
proof of Proposition~\ref{strict monoidal}. See~\cite{MR0299653,MR2863452} for the definition of $\mathbf{Cmd}$.
\end{rem}
\subsection{(Co)monads as (co)monoids in $\mathcal F(\chi)$}
By definition, a pair of morphisms
$$
	\Delta \colon
	(\Sigma, \sigma, \gamma) \lto
	(\Sigma, \sigma, \gamma) \otimes(\Sigma, \sigma,
	\gamma),\quad
	\epsilon \colon (\Sigma, \sigma, \gamma) \lto
	\mathbf 1
$$
is a pair of natural transformations $ \Delta \colon \Sigma
\lto \Sigma \Sigma$ and $ \varepsilon \colon \Sigma
\lto \mathbf 1$ that are compatible
with the distributive laws $\sigma$ and $\gamma$.
This gives us the
following characterisation of comonoids in $\mathcal F(\chi)$.
\begin{prop}\label{chi-comonoids}
A factorisation $(\Sigma, \sigma, \gamma)$ is a comonoid in
$\mathcal F(\chi)$ if and only if $\Sigma$ is part of a comonad and
$\sigma, \gamma$ are distributive laws of comonads.
\end{prop}

Dually, a factorisation $(\Sigma, \sigma, \gamma)$ is a monoid in $\mathcal F(\chi)$ if and only if $\Sigma$ is part of a monad and $\sigma,\gamma$ are mixed distributive laws between monads and comonads.

\begin{cor}
Let $\chi \colon \mathsf{id_\mathcal A} \lto \mathsf{id_\mathcal A}$ be the trivial
distributive law given by the identity.
Then $(T, \Delta, \epsilon)$ is a comonad on $\mathcal A$ if and only if
$(T, \mathsf{id}_T, \mathsf{id}_T)$ is a comonoid in $\mathcal F(\chi)$,
and $(B, \mu, \eta)$ is a monad on $\mathcal A$ if and only if $(B, \mathsf{id}_B, \mathsf{id}_B)$ is a monoid in $\mathcal F(\chi)$.
\end{cor}
\subsection{Module categories for $\mathcal F(\chi)$}\label{modulecat}
We define a functor $\rhd \colon \mathcal F (\chi) \times \mathcal R (\chi) \lto \mathcal R(\chi)$ as follows. On objects we define
$$
(\Sigma, \sigma, \gamma) \rhd (M, \mathcal X, \rho) = (\Sigma M, \mathcal X, \gamma M \circ \Sigma\rho \circ \sigma M)
$$
and on morphisms we define $\alpha \rhd ({\varphi},F)$ to be the pair $(\alpha {\varphi}, F)$.
\begin{prop}\label{right well defined}
The assignment $\rhd$ is a well-defined functor.
\end{prop}
\begin{proof}
Consider the diagram
$$ 
\xymatrix@C=3.5em{
T \Sigma M \ar[d]_-{\sigma M} \ar[r]^{\Delta^T \Sigma M} & TT\Sigma M \ar[r]^-{T \sigma M} & T\Sigma TM \ar[d]^-{\sigma{TM}} \ar[r]^-{T\Sigma\rho} & T \Sigma CM \ar[d]^-{\sigma CM} \ar[r]^-{T \gamma M} & TC\Sigma M \ar[d]^-{\chi \Sigma M} \\
\Sigma TM \ar[dd]_-{\Sigma \rho} \ar[rr]_-{\Sigma \Delta^TM} & & \Sigma TTM \ar[r]_-{\Sigma T\rho}  & \Sigma TCM \ar[d]^-{\Sigma \chi M} & CT\Sigma M \ar[d]^-{C } \\
& & & \Sigma CTM \ar[r]^-{\gamma{TM}} \ar[d]_-{\Sigma C \rho}& C\Sigma TM \ar[d]^-{C \Sigma \rho} \\
\Sigma CM \ar[d]_-{\gamma{CM}}\ar[rrr]_-{\Sigma \Delta^C M}& & & \Sigma CCM \ar[r]_-{\gamma {CM}}&  C \Sigma CM \ar[d]^-{C \gamma M} \\
C\Sigma M \ar[rrrr]_-{\Delta^C {\Sigma M}}& & & & CC \Sigma M
}
$$
The top-left and bottom rectangles commute by the distributive law axioms, the middle-left rectangle commutes because $(M, \mathcal X, \rho)$ is a right $\chi$-coalgebra, the top-right diagram commutes by the Yang-Baxter condition, and the remaining squares commute by naturality of $\sigma,\gamma$. Therefore, the outer rectangle commutes.

Consider the triangle
$$
\xymatrix{
T\Sigma M \ar[rrd]_-{\epsilon^T {\Sigma M}} \ar[r]^-{\sigma M} & \Sigma TM \ar[dr]^-{\Sigma \epsilon^T M} \ar[rr]^-{\Sigma \rho} & & \Sigma CM \ar[dl]_-{\Sigma \epsilon^C M} \ar[r]^-{\gamma M} & C \Sigma M \ar[dll]^-{\epsilon^C {\Sigma M}}\\
& & \Sigma M & &
}
$$
The middle triangle commutes because $(M, \mathcal X, \rho)$ is a right $\chi$-coalgebra, and the other two inner triangles commute by the distributive law axioms. Therefore, the outer triangle commutes. This shows that $\rhd$ is well-defined on objects.

Let $({\varphi},F) \colon (M,\mathcal X, \rho) \lto (M',\mathcal X', \rho')$ and $\alpha \colon (\Sigma, \sigma, \gamma) \lto (\Sigma', \sigma', \gamma')$ be morphisms of right $\chi$-coalgebras and factorisations, respectively. Consider the diagram
$$
\xymatrix{
T \Sigma M \ar[d]_-{\sigma M} \ar[r]^-{T \alpha M} & T \Sigma' M \ar[d]^-{\sigma M'} \ar[r]^-{T \Sigma' {\varphi}} & T \Sigma' M'F \ar[d]^-{\sigma' {M' F}} \\
\Sigma TM \ar[d]_-{\Sigma \rho} \ar[r]^-{\alpha {TM}} & \Sigma' TM \ar[d]^-{\Sigma' \rho} \ar[r]^-{\Sigma' T {\varphi}} & \Sigma' TM'F \ar[d]^-{\Sigma' \rho' F} \\
\Sigma CM \ar[r]_-{\alpha {CM}} & \Sigma' CM \ar[r]_-{\Sigma' C{\varphi}} & \Sigma' C M'F
}
$$
The top-left square commutes since $\alpha$ is
compatible with $T$, the top-right square commutes by
naturality of $\sigma$, the bottom-left square commutes
by naturality of $\alpha$, and the bottom-right square
commutes since $({\varphi},F)$ is a right
$\chi$-coalgebra morphism. Thus the outer square commutes, which shows that $\alpha \rhd ({\varphi},F)$ is a right $\chi$-coalgebra morphism.

It is clear that $\rhd$ respects identities and composition of morphisms (because the vertical and horizontal compositions of natural transformations are compatible with each other), 
so $\rhd$ is well-defined on morphisms.
\end{proof}

Dually, we also define a functor
$$
\lhd \colon \mathcal L(\chi) \times \mathcal F(\chi) \lto \mathcal  L(\chi).
$$

\begin{thm}\label{main theorem}
The category $\mathcal R(\chi)$ is a strict left module category for $\mathcal F(\chi)$, with left action given by the functor $\rhd$. Furthermore, the category $\mathcal L (\chi)$ is a strict right module category for $\mathcal F(\chi)$, with right action given by the functor $\lhd$.
\end{thm}
\begin{proof}
We will prove only the first statement, as the second follows by a similar argument. It is clear that $\mathbf 1$ acts as the identity. Let $(\Sigma, \sigma, \gamma),(\Sigma', \sigma', \gamma')$ be two factorisations and let $(M, \mathcal X, \rho)$ be a right $\chi$-coalgebra. We have
\begin{align*}
&((\Sigma, \sigma, \gamma) \otimes (\Sigma', \sigma', \gamma')) \rhd (M, \mathcal X, \rho)\\
&= (\Sigma \Sigma', \Sigma \sigma' \circ \sigma {\Sigma'}, \gamma {\Sigma'} \circ \Sigma \gamma') \rhd (M,\mathcal X,\rho) \\
&=( \Sigma \Sigma' M, \mathcal X, \gamma {\Sigma' M} \circ \Sigma \gamma' M \circ \Sigma \Sigma'\rho \circ \Sigma \sigma' M \circ \sigma {\Sigma'M})
\end{align*}
and
\begin{align*}
&(\Sigma, \sigma, \gamma) \rhd ((\Sigma', \sigma', \gamma') \rhd (M, \mathcal X, \rho))\\
&= (\Sigma, \sigma, \gamma) \rhd  (\Sigma M, \mathcal X, \gamma M \circ \Sigma\rho \circ \sigma M)\\ 
&=( \Sigma \Sigma' M, \mathcal X, \gamma {\Sigma' M} \circ \Sigma \gamma' M \circ \Sigma \Sigma'\rho \circ \Sigma \sigma' M \circ \sigma {\Sigma'M})
\end{align*}
These are functorially equal, so $\rhd$ is a left action of $\mathcal F(\chi)$.
\end{proof}
\begin{cor}\label{korollar}
The category $\mathcal S(\chi)$ is a strict bimodule category for $\mathcal F(\chi)$.
\end{cor}
\begin{proof}
This follows immediately by applying Lemma~\ref{product bimodule} to Theorem~\ref{main theorem}.
\end{proof}
\section{Examples}
\subsection{Flat connections}\label{flat}

Let $\mathbb B = (B, \mu, \eta)$ be a monad on a category $\mathcal C$. The forgetful functor $U \colon \mathcal C^\mathbb B \lto \mathcal C$ has a left adjoint $F$ defined by
$$
F(X, \alpha) = (BX, \mu_X), \quad \quad \quad F(f) = Bf.
$$
The unit of this adjunction is given by $\eta$ and the counit is $\tilde\epsilon_{(X, \alpha)} = \alpha$. Let $\tilde B$ denote the functor $FU$ and let $\tilde \Delta$ denote the natural transformation $F \eta\, U$. The adjunction gives rise to a comonad $\tilde{\mathbb B} = (\tilde B, \tilde \Delta, \tilde\epsilon)$, which is the same as the comonad discussed in Section~\ref{Eilenberg-Moore}.

Let $\Sigma \colon \C^\mathbb B \lto \C^\mathbb B$ be an endofunctor. For every object $(X, \alpha)$ in $\mathcal C^\mathbb B$ there are natural isomorphisms
$$
\mathcal C^\mathbb B ( \tilde B \Sigma (X,\alpha), \Sigma \tilde B (X, \alpha) ) \cong \mathcal C ( U\Sigma (X, \alpha) , U \Sigma \tilde B (X, \alpha) )
$$
given by the adjunction, so there is a one-to-one correspondence between natural transformations $\sigma \colon \tilde B \Sigma \lto \Sigma \tilde B$  and natural transformations $\nabla \colon U \Sigma \lto U \Sigma \tilde B$. In fact, $\sigma$ is a distributive law if and only if the diagrams
$$
\xymatrix{
U\Sigma \ar[r]^-{\nabla} \ar[d]_-\nabla & U\Sigma \tilde B \ar[d]^-{\nabla \tilde B} \\
U \Sigma \tilde B \ar[r]_-{U \Sigma \tilde \Delta} & U \Sigma \tilde B \tilde B
}
\quad\quad\quad
\xymatrix{
U \Sigma \ar[r]^-\nabla \ar@{=}[dr] & U \Sigma \tilde B \ar[d]^-{U \Sigma \tilde \epsilon} \\
& U \Sigma
}
$$
commute.
\begin{defn}
We say that the natural transformation $\sigma$ is a \emph{connection} if
$\tilde \varepsilon $ is compatible with $\sigma$,
i.e.\ the second diagram above commutes for the
corresponding natural transformation $\nabla$. We say
that a connection $\sigma$ is \emph{flat} if $\tilde
\Delta $ is compatible with $\sigma$, i.e.\ $\sigma$ is a distributive law, or equivalently, both diagrams above commute.
\end{defn}

The terminology is motivated by the special case
discussed in detail in the following section.

\subsection{$(A,A)$-bimodules}\label{connectexample}
Let $k$  be a commutative ring and let $A$ be a unital associative algebra over $k$. Let $\mathcal C = \amod$ be the category of left $A$-modules. The functor $B = - \otimes_k A \colon \C \lto \C$, together with the natural transformations
\begin{align*}
\mu_M \colon M \otimes_k A \otimes_k A & \lto M \otimes_k A & \eta_M \colon M  &\lto M \otimes_k A \\
m \otimes a \otimes b & \lmapsto m \otimes ab & m &\lmapsto m \otimes 1
\end{align*}
defines a monad $\mathbb B$ on $\mathcal C$ which lifts to a comonad
$\tilde {\mathbb B}$ on $\mathcal C^\mathbb B$. The
latter is isomorphic to the
category of $(A,A)$-bimodules (with symmetric action of $k$).

The functor $D= A \otimes_k - \colon \C \lto \C$, together with the natural transformations
\begin{align*}
\Delta_M \colon A \otimes_k M &\lto A \otimes_k A \otimes_k M & \epsilon_M \colon A \otimes_k M & \lto M \\
a \otimes m &\lmapsto  a \otimes 1 \otimes m & a \otimes m & \lmapsto am
\end{align*}
defines a comonad $\mathbb D$ on $\mathcal C$. There is a mixed distributive law $\theta \colon \mathbb B \lto \mathbb D$ given by rebracketing on components
$$
\theta_M \colon (A \otimes_k M) \otimes_k A \lto A \otimes_k (M \otimes_k A)
$$
so this lifts to a comonad distributive law $\theta \colon \tilde{\mathbb B} \lto \tilde{\mathbb D}$.

Let $N$ be an $(A,A)$-bimodule and $\Sigma \colon
\C^\mathbb{B} \lto \C^\mathbb B$ be the functor defined by $\Sigma (M) = M \otimes_A N$. We have that $\Sigma \tilde D = \tilde D \Sigma$ so the identity $\mathsf{id}_{\Sigma \tilde D} \colon \Sigma \lto \mathbb D$ is a distributive law.

In this case, the components of a natural transformation 
$\nabla \colon U \Sigma \lto U \Sigma \tilde B$ are given by a left $A$-linear map
\begin{align*}
  \nabla_M \colon M \otimes_A N & \lto (M \otimes_k A) \otimes_A N \cong M \otimes_k N
\end{align*}
The corresponding natural transformation 
$\sigma \colon \tilde{\mathbb B} \lto \Sigma$ is given by
\begin{align*}
\sigma_M \colon (M \otimes_A N) \otimes_k A &\lto (M \otimes_k A) \otimes_A N \cong M \otimes_k N \\
 (m \otimes_A n) \otimes b &\lmapsto \nabla_M (m \otimes_A n)b.
\end{align*}
The natural transformation $\nabla$ defines a
connection if and only if each $\nabla_M$ splits the
quotient map $M \otimes_k N \lto M \otimes_A N$.
Taking $M=A$ yields an $A$-linear splitting of the action $A
\otimes _k N \lto N$, so $N$ is $k$-relative
projective. Conversely, given a splitting
$n \mapsto n_{(-1)} \otimes n_{(0)}$ of the action, we
obtain $\nabla_M$ as
$\nabla_M(m \otimes_A n)=mn_{(-1)} \otimes n_{(0)}$.

Thus we have:

\begin{prop}
The functor $\Sigma$ admits a connection $ \sigma $ if and only if $N$ is
$k$-relative projective as a left $A$-module.
\end{prop}

Composing $\nabla_A$ with the noncommutative De Rham
differential
$$
	\operatorname d \colon A \lto
	\Omega^1_{A,k},\quad
	a \lmapsto 1 \otimes a-a \otimes 1
$$
gives the notion of connection in noncommutative
geometry \cite[III.3.5]{MR1303779}.

If $N$ is not just $k$-relative projective but
$k$-relative free, i.e.\
$N \cong A \otimes_k V$ as left $A$-modules, for some
$k$-module $V$, then the assignment $\nabla_M ( m \otimes_A (a \otimes v) ) = ma \otimes (1 \otimes v)$ defines a flat connection. 
Thus we have:
\begin{prop}\label{bimoduletwist}
The triple
$(\Sigma, \sigma, \mathsf{id}_{\Sigma \tilde D})$
is a factorisation of $\theta$.
\end{prop}

In particular, let $ \sigma \colon A \lto A$ be an
algebra map and $N=A_\sigma$, the $(A,A)$-bimodule
which is $A$ as a left $A$-module with right action of $a
\in A$ given by right multiplication by $ \sigma
(a)$. 
Then we have 
$\Sigma (M) = M \otimes_A A_\sigma \cong M_\sigma $. 
Since $A_\sigma$ is free as a left $A$-module we get a factorisation 
$(\Sigma, \sigma,\id_{\Sigma \tilde D})$ by Proposition~\ref{bimoduletwist}, 
where $\sigma \colon \tilde{\mathbb B} \lto \Sigma$ is the flat connection defined on components by
\begin{align*}
\sigma_M \colon M_\sigma \otimes_k A \lto (M \otimes_k A)_\sigma \\
m \otimes a \lmapsto m \otimes \sigma(a).
\end{align*}
Note that we use $ \sigma $ to denote both the algebra
map and the flat connection. 

From the general theory developed in
Section~\ref{results} we obtain therefore an action of
the group of endomorphisms of $A$ on the category of
admissible data for $\theta$. In particular, we can act
on the standard cyclic object associated to $A$
\cite{MR823176,MR695483}, which corresponds to
the following admissible datum.
  
Consider $A$ as a functor $A \colon \{ * \} \lto \C^\mathbb B$ from the one-morphism category to the category of $(A,A)$-bimodules. Since $\tilde B A = \tilde D A = A \otimes_k A$ we have a natural transformation $\rho = \mathsf{id}_{A \otimes_k A} \colon \tilde B A \lto \tilde D A$. The triple $(A,  \{*\},\rho)$ is a right $\theta$-coalgebra.

Considering $(A,A)$-bimodules as either left or
right   
$A^e := A \otimes_k A^\op$-modules, we view the zeroth
Hochschild homology as a functor
$H = - \otimes_{A^e} A \colon 
\mathcal C ^\mathbb B \lto k\mbox{-}\mathsf{Mod}$. 
We define a natural transformation 
$\lambda \colon H \tilde D  \lto H \tilde B$ by
\begin{align*}
\lambda_M\colon ( A\otimes_k M) \otimes_{A^e} A &\lto (M \otimes_k A )\otimes_{A^e} A \cong M \\
 (a \otimes m) \otimes_{A^e} b  &\lmapsto mba
\end{align*}
The pair $(H , k\mbox{-}\mathsf{Mod},\lambda)$ is a
left $\theta$-coalgebra, and the duplicial $k$-module 
associated to the admissible datum
$(A,\{*\},\rho,H,k\mbox{-}\mathsf{Mod},\lambda)$ is
indeed the cyclic object defining the cyclic homology
$HC(A)$.

The cyclic homology of the duplicial object associated to the 
admissible datum
 $$
 (\Sigma, \sigma, \mathsf{id}_{\Sigma \tilde D}) 
\rhd (A, \{*\}, \rho, H, k\mbox{-}\mathsf{Mod}, \lambda) = 
(A_\sigma, \{*\}, \rho \circ \sigma_A, H, k\mbox{-}\mathsf{Mod}, \lambda)
 $$
 is $HC^\sigma(A)$, the $\sigma$-twisted cyclic homology of $A$. 
This was first considered in~\cite{MR1943179} and is discussed in 
Section 5.2 of~\cite{MR2803876} in the context of Hopf algebroids.
Thus the action of the category of factorisations
generalises this twisting procedure. 

\subsection{Mixed factorisations}\label{mixed}
Let $ \mathbb B = (B, \mu, \eta)$ be a monad on a category $\C$ and let $\Sigma \colon \C^\mathbb B \lto \C^\mathbb B$ be a functor. In this section, we consider a special case of Section~\ref{flat}: when the functor $\Sigma$ is a lift of a functor $S \colon \mathcal C \lto \mathcal C$, i.e.\ there is a commutative diagram
$$
\xymatrix{
\C^\mathbb B \ar[d]_-U \ar[r]^-\Sigma & \C^\mathbb B \ar[d]^-U \\
\C \ar[r]_-{S} & \C
}
$$
Let $\mathbb D$ be a comonad on $\mathcal C$ and let $\theta \colon \mathbb B \lto \mathbb D$ be a distributive law. Distributive laws $\gamma \colon S \lto \mathbb D$ lift to give distributive laws $\gamma \colon \Sigma \lto \tilde{ \mathbb D}$, and if $\gamma$ is part of a factorisation $(S, \sigma, \gamma)$ of $\theta \colon \mathbb B \lto \mathbb D$ then we get a factorisation $(\Sigma, \sigma, \gamma)$ of $\theta \colon \tilde{\mathbb B} \lto \tilde{\mathbb D}$.

We consider three special cases of this construction.
The distributive laws used therein are
instances of one defined on the category of right $U$-modules, 
where $U$ is a left Hopf algebroid, which is defined
and discussed in~\cite{1}.

\begin{exa}
Suppose that $\sigma \colon \mathbb B \lto \mathbb B$ is a monad morphism
which is compatible with $\theta$; that is $\sigma \colon B \lto B$ is a
natural transformation such that the three diagrams
$$
	\xymatrix{
	BB \ar[d]_-\mu \ar[r]^-{\sigma \sigma} & BB \ar[d]^-\mu \\
	B \ar[r]_-\sigma & B
	}\quad\quad\quad
	\xymatrix{
	\mathsf{id}_\mathcal{C} \ar[dr]_-\eta \ar[r]^-\eta & B \ar[d]^-\sigma \\
& B
	}\quad\quad\quad
	\xymatrix{
	BD \ar[r]^-{\sigma D} \ar[d]_-\theta & BD \ar[d]^\theta \\
	DB \ar[r]_-{D\sigma} & DB
	}
$$
commute. The first two diagrams say that $\sigma \colon
\mathbb B \lto \mathsf{id}_\mathcal C$ is a
distributive law. The triple $(\id_\mathcal C, \sigma,
\mathsf{id}_{SD})$ is a factorisation of $\theta \colon
\mathbb B \lto \mathbb D$, so we get a factorisation $(\Sigma, \sigma,
\mathsf{id}_{\Sigma \tilde D})$ of $\theta \colon \tilde{\mathbb B} 
\lto \tilde{\mathbb D}$. Explicitly, $\Sigma \colon \mathcal C^\mathbb B \lto \mathcal C^\mathbb B$ is given by
$$
	\Sigma (X, \alpha) =
	(X, \alpha \circ \sigma_X), \quad\quad
    \Sigma (f) = f.
$$

Observe that the composition of monad morphisms
corresponds under this assignment to the monoidal
structure in $\mathcal{F}(\theta)$, so when viewing the
monad morphisms as a monoidal category with composition
as tensor product and the identity
$\mathsf{\mathrm{id}}_B$ as unit object, we have:

\begin{prop}
The assignment $ \sigma \lmapsto
(\Sigma,\sigma,\mathsf{id}_{\Sigma\tilde D})$ is a monoidal functor.
\end{prop}

The factorisation given in Proposition~\ref{bimoduletwist} arises in this way.
\end{exa}

\begin{exa}\label{Hopf}
Let $k$ be a commutative ring and let $U$ be a Hopf algebra over $k$. We use Sweedler notation to denote the coproduct
 $$
  \Delta(u) = u_{(1)} \otimes u_{(2)}.
  $$
  See~\cite{MR0252485,MR1243637} for more information about Hopf algebras.
  
Consider the category $\mathcal C =
k\mbox{-}\mathsf{Mod}$. The functor $B = - \otimes_k U
\colon \mathcal C \lto \mathcal C$ is part of a monad
$\mathbb B$ where the multiplication is given by the
multiplication of the algebra $U$ and the unit is given
by the unit of the algebra $U$. Dually, the functor $D = U \otimes_k - \colon \mathcal C \lto \mathcal C$ is part of a comonad, whose structure is given by the comultiplication and counit of the coalgebra $U$. There is a mixed distributive law $\theta \colon \mathbb B \lto \mathbb D$ given by
\begin{align*}
  \theta_X \colon U \otimes_k X \otimes_k U &\lto U \otimes_k X\otimes_k U \\
  u \otimes x \otimes v &\lmapsto S(v_{(2)})u \otimes x
\otimes v_{(1)}.
\end{align*}
Let $P$ be any right $U$-module. This defines a functor $P \otimes_k - \colon \mathcal C \lto \mathcal C$. The maps
\begin{align*}
\sigma_X \colon P \otimes_k X \otimes_k U &\lto P \otimes_k X \otimes_k U \\
p \otimes x \otimes u &\lmapsto pu_{(1)} \otimes x \otimes u_{(2)}
\end{align*}
define a distributive law $\sigma \colon \mathbb B \lto P \otimes_k -$ and the maps
\begin{align*}
\gamma_X \colon P \otimes_k U \otimes_k X &\lto U \otimes_k P \otimes_k X \\
p \otimes u \otimes x &\lmapsto u \otimes p \otimes x
\end{align*}
define a distributive law $\gamma \colon P \otimes_k - \lto \mathbb D$. The triple $(P \otimes -, \sigma, \gamma)$ is a factorisation of $\theta \colon \mathbb B \lto \mathbb D$, and so this gives a factorisation of $\theta \colon \tilde{\mathbb B} \lto \tilde{\mathbb D}$ in the category $\mathcal C^\mathbb B \cong \operatorname{Mod}\mbox{-}U$.
\end{exa}
\begin{exa}\label{Hopf2}
Let $\mathcal C =  k\mbox{-}\mathsf{Mod}$ where $k$ is a commutative ring, and consider the functor $B = U \otimes_k - \colon \mathcal C \lto \mathcal C$. Similarly to Example~\ref{Hopf2}, this is simultaneously part of a monad $\mathbb B$ and a comonad $\mathbb D$. There is a mixed distributive law $\theta \colon \mathbb B \lto \mathbb D$ given by
\begin{align*}
  \theta_X \colon U \otimes_k U \otimes_k X &\lto U \otimes_k U \otimes_k X \\
  u \otimes v \otimes x &\lmapsto vS(u_{(2)})\otimes u_{(1)}\otimes x
\end{align*}
and a distributive law $\tau \colon \mathbb B \lto B$ given by 
\begin{align*}
\tau_X \colon U \otimes_k U\otimes_k X &\lto U \otimes_k U \otimes_k X \\
u \otimes v \otimes x &\lmapsto v \otimes u \otimes x.
\end{align*}
If $U$ is commutative (or even just if the antipode $S$ maps into the centre of $U$), then $(B, \tau, \theta)$ is a factorisation of $\theta \colon \mathbb B \lto \mathbb D$ and so $(\tilde B, \tau, \theta)$ is a factorisation of $\theta \colon \tilde{\mathbb B} \lto \tilde{\mathbb B}$ in $\mathcal C^\mathbb B \cong U\mbox{-}\mathsf{Mod}$.
\end{exa}

\subsection{Braided distributive laws}\label{braided} 
Let $\chi \colon \mathbb T \lto \mathbb C$ be a comonad distributive law on a category $\mathcal A$.
\begin{defn}
A distributive law $\tau \colon \mathbb T \lto T$ between the comonad $\mathbb T$ and the endofunctor $T$ is \emph{braided with respect to $\chi$} if the hexagon
$$
\xymatrix@R=0.5em{
&  TTC \ar[r]^-{T \chi} & T CT \ar[dr]^-{\chi T} &\\
TT C \ar[ur]^-{\tau C} \ar[dr]_-{T\chi} & & & C T T \\
& TCT \ar[r]_-{\chi T} & CTT \ar[ur]_-{C \tau} &
}
$$
commutes. Dually, we say that a distributive law 
$\phi \colon C \lto \mathbb C$ between the endo\-functor $C$ and the comonad $\mathbb C$ is \emph{braided with respect to $\chi$} if a similar hexagon commutes.
\end{defn}
Clearly, $\tau$ is braided if and only if $(T, \tau, \chi)$ is a factorisation of $\chi$, since the above hexagon is just the Yang-Baxter condition in that case. In the dual case, $(C, \chi, \phi)$ would be a factorisation of $\chi$.

\begin{exa}
In Example~\ref{Hopf2}, the distributive law $\tau$ is braided with respect to $\theta$.
\end{exa}
\begin{exa}
Let $\tau \colon \mathbb T \lto \mathbb T$ be a BD-law. These are defined in~\cite{MR2116327} and are exactly those distributive laws which are braided with respect to themselves. Thus $(T, \tau, \tau)$ is a factorisation of $\tau$.
\end{exa}

\begin{exa}\label{braidd}
For this example we relax the assumption that
monoidal categories are strict.  
Let $\mathcal A$ be a braided monoidal category with tensor product 
$\otimes$, associator morphisms $\alpha$ and braiding morphisms $b$. Let $\mathbf U= (U, \Delta^U, \epsilon^U)$ and $\mathbf V = (V, \Delta^V, \epsilon^V)$ be comonoids in $\mathcal A$. The comonoids $\mathbf U, \mathbf V$ define two comonads $\mathbb U, \mathbb V$ with endofunctors $U \otimes -, V \otimes -$ respectively, and three distributive laws $\chi \colon \mathbb U \lto \mathbb V$, $\tau\colon \mathbb U \lto \mathbb U$  and $\phi \colon \mathbb V \lto \mathbb V$ defined by
\begin{align*}
\xymatrix@C=3em{
U \otimes (V \otimes X)\ar[r]^-{\alpha^{-1}_{U,V,X}} & (U \otimes V) \otimes X\ar[r]^{b_{U,V} \otimes \mathsf{id}} & (V \otimes U) \otimes X \ar[r]^-{\alpha_{V,U,X}} & V \otimes (U \otimes X)
}\\
\xymatrix@C=3em{
U \otimes (U \otimes X)\ar[r]^-{\alpha^{-1}_{U,U,X}} & (U \otimes U) \otimes X\ar[r]^{b_{U,U} \otimes \mathsf{id}} & (U \otimes U) \otimes X \ar[r]^-{\alpha_{U,U,X}} & U \otimes (U \otimes X)
}\\
\xymatrix@C=3em{
V \otimes (V \otimes X)\ar[r]^-{\alpha^{-1}_{V,V,X}} & (V \otimes V) \otimes X\ar[r]^{b_{V,V} \otimes \mathsf{id}} & (V \otimes V) \otimes X \ar[r]^-{\alpha_{V,V,X}} & V \otimes (V \otimes X)
}
\end{align*}
respectively. The distributive laws $\tau$ and $\phi$ are both braided with respect to $\chi$ so we get two factorisations $(U \otimes -, \tau, \chi)$ and $(V \otimes -, \chi, \phi)$ of $\chi$. By Proposition~\ref{chi-comonoids} these are both comonoids in $\mathcal F(\chi)$. This example comes from the dual of Example 1.11 in~\cite{MR2501178}.
\end{exa}

\subsection{Quantum doubles}\label{qd}
In our final example, we consider the distributive laws
corresponding to quantum doubles: let $B$ and $C$ be
two Hopf algebras over a commutative ring $k$ and 
$\mathcal{R} \in C \otimes_k B$ be an invertible
2-cycle, meaning that we have
$$
	(\Delta^C \otimes_k \id_B) (\mathcal R)=
	\mathcal R_{13}\mathcal R_{23},\quad
	(\id_C \otimes_k \Delta^B)(\mathcal R)=\mathcal
	R_{12} \mathcal R_{13},	
$$ 
$$
	(\id_C \otimes_k S^B)(\mathcal{R})=\mathcal
	R^{-1},\quad
	(S^C \otimes_k \id_B)(\mathcal R)=\mathcal R^{-1},
$$
where $\mathcal R^{-1}$ refers to the multiplicative
inverse in the tensor product algebra $C \otimes_k B$
and subscripts denote components in $C \otimes_k C
\otimes_k B$ respectively $C \otimes_k B \otimes_k B$. 
We refer to~\cite[]{MR1358358} for more background
information.  

The coalgebras $B$ and $C$ define comonads $\mathbb T$
and $\mathbb C$ on
 $\mathcal A=k\mbox{-}\mathsf{Mod}$ given by $B \otimes_k -$ and 
$C \otimes_k -$ with structure maps given by the
coproducts and the counits. The 2-cycle $\mathcal R$
defines a distributive law $ \chi \colon  
\mathbb T \lto \mathbb C$ given by

\begin{align*}
	\chi _X \colon B \otimes_k C \otimes_k X 
	& \lto 
	C \otimes_k B \otimes_k X \\
	b \otimes c \otimes x & \lmapsto 
	\mathcal R (c \otimes b) \mathcal R^{-1} \otimes x.
\end{align*} 

In this case, every $(B,C^\op)$-bimodule $M$, that is,
a $k$-module $M$ with two commuting left actions of $B$
and $C$, gives rise to a factorisation of $ \chi $: 
let $ \Sigma \colon \mathcal A \lto \ma$ be the functor
$M \otimes_k -$. We define distributive laws 
\begin{align*}
	\sigma_X \colon B \otimes_k M \otimes_k X &\lto 
	M \otimes_k B \otimes_k X,&
	\gamma_X \colon M \otimes_k C \otimes_k X &\lto
	C \otimes_k M \otimes_k X,\\ 
	b \otimes m \otimes x &\lmapsto 
	\mathcal R_{12} (m \otimes b \otimes x),&
	m \otimes c \otimes x &\lmapsto
	\mathcal R_{12}(c \otimes m \otimes x).
\end{align*}
Then a straightforward computation shows that 
$(\Sigma,\sigma,\gamma)$ is a factorisation of $ \chi
$. When $M=k$ with trivial actions given by the
counits, we recover Example~\ref{braidd}.


\begin{thebibliography}{KMT03}

\bibitem[Bec69]{MR0241502}
Jon Beck.
\newblock Distributive laws.
\newblock In {\em Sem. on {T}riples and {C}ategorical {H}omology {T}heory
  ({ETH}, {Z}\"urich, 1966/67)}, pages 119--140. Springer, Berlin, 1969.

\bibitem[BLS11]{MR2863452}
Gabriella B{\"o}hm, Stephen Lack, and Ross Street.
\newblock On the 2-categories of weak distributive laws.
\newblock {\em Comm. Algebra}, 39(12):4567--4583, 2011.

\bibitem[B{\c{S}}08]{MR2415479}
Gabriella B{\"o}hm and Drago{\c{s}} {\c{S}}tefan.
\newblock ({C}o)cyclic (co)homology of bialgebroids: an approach via
  (co)monads.
\newblock {\em Comm. Math. Phys.}, 282(1):239--286, 2008.

\bibitem[B{\c{S}}09]{MR2501178}
Gabriella B{\"o}hm and Dragos {\c{S}}tefan.
\newblock Examples of para-cocyclic objects induced by {\it {bd}}-laws.
\newblock {\em Algebr. Represent. Theory}, 12(2-5):153--180, 2009.

\bibitem[Bur73]{MR0323864}
{\'E}lisabeth Burroni.
\newblock Lois distributives mixtes.
\newblock {\em C. R. Acad. Sci. Paris S\'er. A-B}, 276:A897--A900, 1973.

\bibitem[Bur09]{MR2520969}
Elisabeth Burroni.
\newblock Lois distributives. {A}pplications aux automates stochastiques.
\newblock {\em Theory Appl. Categ.}, 22:No. 7, 199--221, 2009.

\bibitem[CP95]{MR1358358}
Vyjayanthi Chari and Andrew Pressley.
\newblock {\em A guide to quantum groups}.
\newblock Cambridge University Press, Cambridge, 1995.
\newblock Corrected reprint of the 1994 original.

\bibitem[Con85]{MR823176}
Alain Connes.
\newblock Noncommutative differential geometry.
\newblock {\em Inst. Hautes \'Etudes Sci. Publ. Math.}, (62):257--360, 1985.

\bibitem[Con94]{MR1303779}
Alain Connes.
\newblock {\em Noncommutative geometry}.
\newblock Academic Press, Inc., San Diego, CA, 1994.


\bibitem[DK85]{MR826872}
W.~G. Dwyer and D.~M. Kan.
\newblock Normalizing the cyclic modules of {C}onnes.
\newblock {\em Comment. Math. Helv.}, 60(4):582--600, 1985.

\bibitem[KK11]{MR2803876}
Niels Kowalzig and Ulrich Kr{\"a}hmer.
\newblock Cyclic structures in algebraic (co)homology theories.
\newblock {\em Homology Homotopy Appl.}, 13(1):297--318, 2011.

\bibitem[KKS]{1}
Niels Kowalzig, Ulrich Kr{\"a}hmer, and Paul Slevin.
\newblock In preparation.

\bibitem[KLV04]{MR2116327}
Stefano Kasangian, Stephen Lack, and Enrico~M. Vitale.
\newblock Coalgebras, braidings, and distributive laws.
\newblock {\em Theory Appl. Categ.}, 13:No. 8, 129--146, 2004.

\bibitem[KMT03]{MR1943179}
J.~Kustermans, G.~J. Murphy, and L.~Tuset.
\newblock Differential calculi over quantum groups and twisted cyclic cocycles.
\newblock {\em J. Geom. Phys.}, 44(4):570--594, 2003.

\bibitem[Lod08]{MR2504663}
Jean-Louis Loday.
\newblock Generalized bialgebras and triples of operads.
\newblock {\em Ast\'erisque}, (320):x+116, 2008.

\bibitem[ML98]{MR1712872}
Saunders Mac~Lane.
\newblock {\em Categories for the working mathematician}, volume~5 of {\em
  Graduate Texts in Mathematics}.
\newblock Springer-Verlag, New York, second edition, 1998.

\bibitem[Mon93]{MR1243637}
Susan Montgomery.
\newblock {\em Hopf algebras and their actions on rings}, volume~82 of {\em
  CBMS Regional Conference Series in Mathematics}.
\newblock Published for the Conference Board of the Mathematical Sciences,
  Washington, DC; by the American Mathematical Society, Providence, RI, 1993.

\bibitem[Str72]{MR0299653}
Ross Street.
\newblock The formal theory of monads.
\newblock {\em J. Pure Appl. Algebra}, 2(2):149--168, 1972.

\bibitem[Swe69]{MR0252485}
Moss~E. Sweedler.
\newblock {\em Hopf algebras}.
\newblock Mathematics Lecture Note Series. W. A. Benjamin, Inc., New York,
  1969.

\bibitem[Tsy83]{MR695483}
B.~L. Tsygan.
\newblock Homology of matrix {L}ie algebras over rings and the {H}ochschild
  homology.
\newblock {\em Uspekhi Mat. Nauk}, 38(2(230)):217--218, 1983.

\bibitem[Tur96]{MR1692751}
Daniele Turi.
\newblock {\em Functorial operational semantics and its denotational dual}.
\newblock Vrije Universiteit te Amsterdam, Amsterdam, 1996.
\newblock Dissertation, Vrije Universiteit, Amsterdam, 1996.

\bibitem[VW06]{MR2220892}
Daniele Varacca and Glynn Winskel.
\newblock Distributing probability over non-determinism.
\newblock {\em Math. Structures Comput. Sci.}, 16(1):87--113, 2006.

\end{thebibliography}
\end{document}